\documentclass[11pt]{article}
\usepackage{amsmath,amsthm,amsfonts,amssymb,amscd}
\usepackage{bbm}
\usepackage{url}
\usepackage{cite}
\usepackage{fullpage}
\usepackage{fancyhdr}
\usepackage{amsmath,amsthm,amsfonts,amssymb,amscd}
\usepackage{color}
\usepackage{float}
\usepackage{soul}
\usepackage{graphicx}
\usepackage[margin=1in]{geometry}

\theoremstyle{plain}
\newtheorem{theorem}{Theorem}[section]
\newtheorem{lemma}[theorem]{Lemma}

\newtheorem{proposition}[theorem]{Proposition}
\newtheorem{definition}{Definition}[section]

\newtheorem{example}{Example}[section]

\newtheorem*{AlgA}{Algorithm A}
\newtheorem*{AlgA1}{Algorithm A1}
\newtheorem*{AlgA2}{Algorithm A2}
\newtheorem*{AlgA3}{Algorithm A3}
\newtheorem*{Inner}{Inner Loop}

\usepackage{cite}
\theoremstyle{definition}

\newcommand{\argmin}{{\rm arg}\!\min}

\newcommand{\la}{\langle}
\newcommand{\ra}{\rangle}

\newcommand{\nexto}{\kern -0.54em}

\newcommand{\dZ}{{\cal Z \kern -0.7em Z}}
\newcommand{\dC}{{\rm\hbox{C \kern-0.8em\raise0.2ex\hbox{\vrule height5.4pt width0.7pt}}}}
\newcommand{\dQ}{{\rm\hbox{Q \kern-0.85em\raise0.25ex\hbox{\vrule height5.4pt width0.7pt}}}}

\usepackage{epsfig}
\usepackage{latexsym}

\newcommand{\NN}{\mathbb{N}}
\newcommand{\HH}{\mathcal{H}}
\newcommand{\RR}{\mathbbm{R}}

\newcommand{\dsty}{\displaystyle}

\begin{document}

\title{A Relaxed-Projection Splitting Algorithm for Variational Inequalities in Hilbert Spaces}
\author{
J.Y. Bello Cruz\footnote{Institute of Mathematics and Statistics,
Federal University of Goi\'as, Goi\^ania, Avenida Esperança, s/n Campus Samambaia,  Goi\^ania, GO, 74690-900, Brazil.
E-mail: yunier@impa.br}\and R. D\'iaz Mill\'an\footnote{Mathematics,
Federal Institute of Education, Science and Technology, Goi\^ania,
G.O. 74055-110, Brazil.  E-mail: rdiazmillan@gmail.com}}

\maketitle

\begin{abstract}
\noindent We introduce a relaxed-projection splitting algorithm for
solving variational inequalities in Hilbert spaces for the sum of
nonsmooth maximal monotone operators, where the feasible set is
defined by a nonlinear and nonsmooth continuous convex function
inequality. In our scheme, the orthogonal projections onto the
feasible set are replaced by projections onto separating
hyperplanes. Furthermore, each iteration of the proposed method
consists of simple subgradient-like steps, which does not demand the
solution of a nontrivial subproblem, using only individual
operators, which exploits the structure of the problem. Assuming
monotonicity of the individual operators and the existence of
solutions, we prove that the generated sequence converges weakly to
a solution.

\medskip


\medskip

\end{abstract}

\section{Introduction}

Let $\HH$ be a real Hilbert space with inner product
$\la\cdot,\cdot\ra$ and induced norm $\|\cdot\|$. For $C$, a
nonempty, convex and closed subset of $\HH$, we define the
orthogonal projection of $x$ onto $C$, $P_C(x)$, as the unique
point in $C$ such that $\| P_C(x)-x\| \le \|y-x\|$ for all $y\in C$.
Recall that an operator $T: \HH \rightrightarrows \HH$ is monotone
if, for all $(x,u),(y,v)\in \mbox{Gr}(T)$, we have $\la x-y, u-v \ra
\ge 0,$ and it is maximal if $T$ has no proper monotone extension in
the graph inclusion sense.

In this paper, we present a relaxed-projection splitting algorithm
for solving the variational inequality problem for $T$ and $C$,
where $T$ is a sum of $m$ nonsmooth maximal monotone operators, i.e,
$T=T_1+T_2+\cdots+ T_m$ with $T_i: \HH \rightrightarrows \HH$,
$(i=1,2,\ldots,m)$ and $C$ is given of the following form: $
C:=\{x\in \HH\,:\,c(x)\leq0\}$ with $c:\HH\rightarrow \RR$ is a
continuous and convex function, possibly nondifferentiable.

It is clear that, if $T_i$, ($i=1,\ldots,m$) are monotone, then
$T=T_1+T_2+\cdots+ T_m$ is also monotone. But if $T_i$,
($i=1,2,\ldots,m$) are maximal, it does not necessarily imply that
$T$ is maximal even when $\mbox{dom}(T)$ is nonempty, additional
conditions are needed, since for example the graph of $T$ can be
even empty (as happens when $\mbox{dom}(T)=\mbox{dom}(T_1)\cap
\mbox{dom}(T_2)\cap \cdots \cap \mbox{dom}(T_m)=\emptyset$).

The problem of determining conditions under which the sum is
maximal, turns out to be of fundamental importance in the theory of
monotone operators. Results in this directions were proved in
\cite{Rock}. It is clear that in our case ($\mbox{dom}(T_i)=\HH$,
$i=1,2,\ldots,m$) all these sufficient conditions for establishing
the maximality of $T$ are satisfied.

\noindent Now, we recall the formulation of the variational
inequality problem for $T$ and $C$, namely:
\begin{equation}\label{problem}
\mbox{Find} \ \ x_*\in C \ \ \mbox{such that} \ \ \exists\, u_*\in
T(x_*),\; \mbox{with} \ \ \la u_*, x-x_*\ra\geq 0 \quad \quad
\forall \,x\in C.
\end{equation}
The solution set of problem \eqref{problem} is denoted by $S_*$.
This problem has been a classical subject in economics, operations
research and mathematical physics, particularly in the calculus of
variations associated with the minimization of infinite-dimensional
functionals; see, for instance, \cite{stamp-kinder} and the
references therein. It is closely related to many problems of
nonlinear analysis, such as optimization, complementarity and
equilibrium problems and finding fixed points; see, for instance,
\cite{peterson, todd, stamp-kinder}. Many methods have been proposed
to solve problem \eqref{problem}, for $T$ point-to-point operator;
see \cite{IS,kor,sol-svaiter-1999,sol-tseng}, and for $T$
point-to-set operator; see \cite{bao,iusem-lucambio,kono,he}. An
excellent survey of methods for variational inequality problems can
be found in \cite{facch-pang-2}.

Variational inequality problems are related to inclusion problems.
In fact, when the feasible set is the whole space, the variational
inequality problem may be formulated as an inclusion problem. In
this work, we are interested in methods that exploit the structure
of $T$. This kind of methods are called splitting, since each
iteration involves only the individual operators, but not the sum.
Many splitting algorithms have been proposed in order to solve this
kind of inclusion problem; see \cite{attouch, comb-2012,pasty,
tseng, lion, moudafi, cheng, eckstein-svaiter-2009} and the
references therein. However, in all of them, the resolvent operator
of any individual operator, has to be evaluated in each iteration. It
is important to mention, that this proximal-like iteration (the evaluation of the resolvent operator) is, in general, a
nontrivial subproblem, which demands hard work from computational
point of view. Our algorithm tries to avoid this task, replacing
proximal-like iterations by explicit subgradient-like steps. This
represents a significant advantage in both implementational and
theoretical senses. 

Another weakness which appears in most of methods proposed in the literature for solving
problem \eqref{problem}, is the necessity to compute the exact
projection onto the feasible set. This limits the applicability of the methods,
especially when such projection is computationally hard to
compute. It is well known that only in a few specific instances the
projection onto a convex set has an explicit formula. When the
feasible set of problem \eqref{problem} is a general closed convex
set, $C$, we must solve a nontrivial quadratical problem, in order
to compute the projection onto $C$. This difficulty also appears
when the feasible set of problem \eqref{problem} is expressed as the
solution set of another problem, as in this paper. In this kind of
problems, it is very hard to find the projection onto the feasible
set or even find a feasible point. One option for avoiding this
difficulty, consists in replacing at each iteration, the projection
onto $C$, by the projection onto halfspaces containing the given set
$C$ and not the current point. For variational inequality problems,
the above approach was introduced in \cite{fukushima}, for
point-to-point and strongly monotone operators assuming a coerciveness condition. Other schemes have
been proposed, in order to improve the convergence results, without
doing exact projections onto $C$: in
\cite{yunier-iusem,yunier-iusem-2} for point-to-set and paramonotone
operators; in \cite{yunier-iusem-3} for point-to-point and monotone
operators, and in \cite{kono1} for point-to-set and monotone operators assuming existence of a Slater point in the feasible set. Algorithms using similar ideas may be found in
\cite{regina-svaiter-2005, censor-gibali-2011}.

In this work, we propose a new splitting scheme for solving problem
\eqref{problem}, in which the orthogonal projections onto the
feasible set, are replaced by projections onto separating
hyperplanes and only simple subgradient-like steps are performed.
Assuming maximal monotonicity of the individual operators and
existence of solutions, we establish the weak convergence to a
solution of the whole generated sequence. 

Our method was inspired by the incremental subgradient method for
nondifferentiable optimization, proposed in \cite{nedic}, and it
uses similar ideas exposed in \cite{yunier-iusem-2,yunier-iusem, yunier-reinier}. In the case of only one operator, it is known that the natural extension of the subgradient iteration (one step) fails in general for monotone operators; see
\cite{yunier-iusem, yunier-iusem-3}. However, as will be shown, adding an ergodic extra step on the spirit of \cite{bruck} the weak convergence of
the sequence generated by our algorithm is proved. Moreover, without assuming existence of Slater points in the feasible set, as is demanding in \cite{kono1}, the proposed scheme relaxes the projection onto the feasible set. Also the inner loop replacing the projection onto the feasible set, may generates un-feasible sequence.  Furthermore, we fully avoid to solve hard subproblems, as find an approximate minimal norm vector in the image of the operator as happens in \cite{kono1}. Finally, our scheme extends the incremental subgradient iteration for restricted and nonsmooth variational inequality problem, enlarging the use of the iteration to a wide classes of problems. It is important to mention that incremental techniques have been recently widely used in the literature; and randomized versions of it may be used also in this setting, exploiting the separating structure of $T$. These ideas were already implemented for optimization problems and now the proposed scheme here uncovers many splitting ideas from optimization, which would be used in variational inequalities.



The paper is organized as follows. The next section provides some
notation and preliminary results that will be used in the remainder
of this paper. The relaxed splitting method is presented in Section
\ref{el-algoritmo} and Subsection \ref{convergencia} contains the
convergence analysis of the algorithm. Section \ref{remarks}
contains some discussion on the assumptions with examples showing
the effectiveness of our scheme.
\section{Preliminary Result}\label{preliminary}
In this section, we present some definitions and results that are
needed for the convergence analysis of the proposed method. First,
we state two well known facts on orthogonal projections.
\begin{lemma}\label{popiedades_projeccion}
Let $S$ be any nonempty closed and convex set in $\HH$, and $P_S$ be
the orthogonal projection onto $S$. For all $x,y\in \HH$ and $z\in
S$, the following properties hold:
\item [ {\bf (a)}] $\dsty \| P_S(x)-P_S(y)\|\,\leq\|
x-y\|$. \item [ {\bf (b)}] $\la x-P_S(x), z-P_S(x)\ra\leq0$.
\end{lemma}
\begin{proof}
See Lemmas $1.1$ and $1.2$ in \cite{zarantonelo}.
\end{proof}
\noindent We next deal with the so called quasi-Fej\'er convergence and its
properties.
\begin{definition}\label{def-cuasi-fejer}
Let $S$ be a nonempty subset of $\HH$. A sequence $(x^k)_{k\in \NN}$
in $\HH$ is said to be quasi-Fej\'er convergent to $S$ if and only
if for all $x \in S$ there exist $k_0\geq 0$ and a sequence
$(\delta_k)_{k\in \NN}$ in $\RR_+$ such that
$\sum_{k=0}^\infty\delta_k<\infty$ and $\| x^{k+1}-x\|^2 \leq \|
x^{k}-x\|^2 +\delta_k,$ for all $k\geq k_0$.
\end{definition}
\noindent This definition originates in \cite{Ermolev} and has been
elaborated further in \cite{IST,comb-2001}.
\begin{proposition}\label{cuasi-Fejer}
If $(x^k)_{k\in \NN}$ is quasi-Fej\'er convergent to $S$, then:
\item [ {\bf (a)}] The sequence $(x^k)_{k\in \NN}$ is bounded.
\item [ {\bf (b)}] If all weak cluster points of $(x^k)_{k\in \NN}$ belong to $S$, then the sequence $(x^k)_{k\in \NN}$
is weakly convergent.
\end{proposition}
\begin{proof} See Theorem $4.1$ in \cite{IST}.\end{proof}
The next lemma will be useful for proving that the sequence
generated by our algorithm, converges weakly to some point belonging to
the solution set of problem \eqref{problem}.
\begin{lemma}\label{distancia-q-F-convergente}
If a sequence $(x^k)_{k\in \NN}$ is quasi-Fej\'er convergent to a
closed and convex set $S$, then the sequence $(P_S(x^k))_{k\in \NN}$
is strongly convergent.
\end{lemma}
\begin{proof}
See Lemma $2$ in \cite{yunier-iusem-2}.
\end{proof}
We also need the following results on maximal monotone operators and
monotone variational inequalities.
\begin{lemma}\label{cond_limitacion}
Let $T:\HH \rightrightarrows \HH$ be a maximal monotone operator and
$C$ be nonempty, closed and convex set in $\HH$. Then, the solution
set of problem \eqref{problem}, $S_*$, is closed and convex.
\end{lemma}
\begin{proof}
See Lemma $2.4$(ii) in \cite{yunier-iusem-1}.
\end{proof}
The next lemma will be useful for proving that all weak cluster
points of the sequence generated by our algorithm belong to the
solution set of problem \eqref{problem}.
\begin{lemma}\label{consecuencia-continuidad}
Consider the variational inequality problem for $T$ and $C$. If
$T:\HH \rightrightarrows \HH$ is maximal monotone, then $ S_*=\{x\in
C \,:\,\la v,y-x\ra\geq0\,,\, \forall \,y\in C, \;\; \forall\, v\in
T(y)\}.$
\end{lemma}
\begin{proof}
See Lemma $3$ in \cite{soldebil}.
\end{proof}
The next lemma provides a computable upper bound for the distance
from a point to the feasible set $C$.
\begin{lemma}\label{conjetura}
Let $c:\HH\rightarrow\RR$ be a convex function and $C:=\{x\in
\HH\,:\,c(x)\le 0\}$. Assume that there exists $w\in C$ such that
$c(w)<0$. Then, for all $y\in \HH$ such that $c(y)>0$, we have
$${\rm dist }(y,C)\le \frac{\|y-w\|}{c(y)-c(w)}\,c(y)\,.$$
\end{lemma}
\begin{proof}
See Lemma $4$ in \cite{yunier-iusem-2}.
\end{proof}
Now, the next  proposition will be useful for calculating the
projections onto the halfspaces that will appear in our algorithm.
\begin{proposition}\label{proj_C(x)}
Let  $c:\HH\to \RR$  be a convex function. Given $y, z, w\in \HH$ and $v \in \partial c(z)$, define
$W_{z,w}:=\{x\in \HH\,:\,\la x-z,w-z \ra\le 0\}$ and $C_z:=\{x\in
\HH\,:\,c(z)+\la v,x-z\ra\leq0\}$. Then,
$$P_{C_z \cap W_{z,w} }(w)=w+ \max\{0,\lambda_1\}v+\lambda_2(w-z),$$where $\lambda_1, \lambda_2$ are solution of the linear system:
\begin{align*}\lambda_1\|v\|^ 2+\lambda_2\la v,w-z\ra&=-\la v, w-z\ra-c(z)\\\lambda_1\la v, w-z\ra+\lambda_2\|w-z\|^2&=-\|w-z\|^2\end{align*}
and $\dsty P_{C_z}(y)= y-\max\left\{0,\frac {c(z)+\la v,y-z\ra\,}{\|
v\|^2}\right\}\,v. $
\end{proposition}
\begin{proof}
See Proposition $28.19$ in \cite{bauschke}.
\end{proof}
\noindent Finally, we will need the following elementary result on
sequence averages.
\begin{proposition}\label{conv-convinacion-de-terminos}
Let $(p^k)_{k\in \NN}\subset\HH$ be a sequence strongly convergent
to $\tilde{p}$. Take nonnegative real numbers $(\zeta_{k,j})_{k\in
\NN,\, 0\le j\le k}$ such that $\lim_{k\rightarrow\infty}
\zeta_{k,j}=0$ for all $0\le j\le k$ and $\sum_{j=0}^k\zeta_{k,j}=1$
for all $k\in \NN$. Define
$$
w^k:=\sum_{j=0}^k \zeta_{k,j}p^j.
$$
Then $(w^k)_{k\in \NN}$ also converges strongly to $\tilde{p}$.
\end{proposition}
\begin{proof}
See Proposition $3$ in \cite{yunier-iusem-2}.
\end{proof}
\section{A Relaxed-Projection Splitting Algorithm}\label{el-algoritmo}
In this section, we introduce an algorithm for solving the
variational inequality problem when \newline $T=T_1+T_2+\ldots+T_m$
and $C$ is of the form
\begin{equation}\label{C}
C=\{x\in \HH\,:\,c(x)\leq0\},
\end{equation}
where $c:\HH\rightarrow\RR$ is a continuous convex function. Differentiability of $c$ is not assumed and therefore the
representation (\ref{C}) is rather general, since any system of
inequalities $c_j(x)\leq0$ with $j\in J$, where each function $c_j$ is continuous and 
convex, may be represented as in (\ref{C}) with
$c(x)=\sup\{c_j(x)\,:\,j\in J\}$. 

Observe that except in very special
cases (e.g. when $C$ is a halfspace, or a ball, or a subspace, or
a box) the exact calculation of the orthogonal projection is a
computationally nontrivial task. Since most of the direct method for solving restricted variational inequalities use exact orthogonal projections, following the ideas proposed by Fukushima in \cite{fukushima} and later in \cite{yunier-iusem-2}, we introduce an inner loop replacing the projection onto $C$ by an unfeasible point.

\begin{center}\fbox{\begin{minipage}[b]{\textwidth}
\begin{Inner}\label{boundary} Given $z$, $\theta$ and $\alpha$.\\
{\bf Input.} Set $y^{0}=z$ and $j=0$. Given $j$ compute
\begin{equation}\label{ykj} y^{j+1}:=P_{C_{j}\cap
W_{j}}(y^{0}),\end{equation}
\begin{equation}\label{Ck}
C_{j}:=\{x\in \HH\,:\,c(y^{j})+\la g^{j},x-y^{j}\ra\leq0\} \qquad g^{j}\in \partial c(y^{j}),
\end{equation}
\begin{equation}\label{Wkj}
W_{j}:=\{x\in \HH\,:\,\la x-y^{j},y^{0}-y^{j}\ra\leq0\}.
\end{equation}
\quad  {\bf While} $
{\rm dist}(y^{j+1}, C)> \theta \alpha$ \;  {\bf do} \;
$j=j+1$  {\rm and} $y^{j+1}=y^{j}$.

\medskip

\quad {\bf End While}

\noindent {\bf Output.} $y^{j+1}$ {\rm and} $C_{j}$.
\end{Inner}\end{minipage}}\end{center}


\noindent Consider an exogenous sequence $(\alpha_k)_{k\in \NN}$ in
$\RR_{++}$. Then the algorithm is defined as follows.

\begin{center}\fbox{\begin{minipage}[b]{\textwidth}
\begin{AlgA}\label{A1} Given $\theta>0$.

\item []\noindent  {\bf Initialization Step.} Take $ x^0\in \HH$  and define $z^0:=x^0$ and $\sigma_0:=\alpha_0$.

\item [] \noindent {\bf Iterative Step.} Given $z^k$. If $c(z^k)\le 0$, then
take $z^k_0:=z^k$ and $C_k=\{x \in \HH: \la g^k, x-z^{k}_{0} \ra\leq
0\}$ where $g^k \in \partial c^+(z^{k}_{0})$ with
$c^+(x)=\max\{0,c(x)\}$. Else, perform {\bf Inner Loop\,($z^k,\theta,\alpha_k$)=:($z_0^k,C_k$)}. 
Compute the cycle, from $i=1,2,\ldots,m$, as follows
\begin{equation}\label{algoritmo_1_paso_2}
z^{k}_i=P_{C_{k}}\left(z^{k}_{i-1}-\alpha_k u^{k}_{i}\right),
\end{equation}
where $u_{i}^{k}\in T_i(z^{k}_{i-1})$.
 The vector $z^{k+1}$ is obtained doing $ z^{k+1}=z^{k}_m$. Define
\begin{equation}\label{16'}\sigma_k:=\sigma_{
k-1}+\alpha_k,\end{equation} and
\begin{equation}\label{algoritmo_2_paso_3-bonito}
x^{k+1}:=\left(1-\frac{\alpha_k}{\sigma_k}\right)x^{k}+\frac{\alpha_k}{
\sigma_k} z^{k+1}.
\end{equation}
\end{AlgA}\end{minipage}}\end{center}
Before the formal analysis of the convergence properties of
{\bf Algorithm A}, we make some comments and discuss about our
assumptions. 

First, unlike other projection methods, {\bf Algorithm A}
generates a sequence $(x^k)_{k\in \NN}$ which is not necessarily
contained in the set $C$. As will be shown in the next subsection,
the generated sequence is asymptotically feasible and, in fact,
converges to some point in the solution set. We observe that in {\bf {\bf Algorithm A}}, {\bf Inner Loop} starts with the point $z^k$
and ends with a point $z^{k}_{0}$ close to $C$, in fact ${\rm
dist}(z^{k}_{0},C)\le \theta \alpha_k$, this is possible since {\bf Inner Loop}, in the step $k$ of our algorithm, is a direct application of the
algorithm proposed in \cite{yunier-iusem-4}, with $C=\HH$,
$x^0=y^{k,0}$, $f(x)=c^+(x):=\max\{0,c(x)\}$ and $f_*:=\inf_{x\in
\HH}c^+(x)=0$. Recently has been proposed in \cite{yunier-wlo} a
restricted memory level bundle method improving the convergence
result of \cite{yunier-iusem-4}, which may be used into the inner
loop accelerating its convergence.

It might seem that this inner loop can be replaced by any finite
procedure leading to an approximation of $P_C(z^k)$, say a point
$z^{k}_{0}$ such that $\|z^{k}_{0}-P_C(z^k)\|$ is sufficiently small. This
is not the case: In the first place, depending on the location of
the intermediate hyperplanes $C_{k,j}$ and $W_{k,j}$, the sequence
$(y^{k,j})_{j\in \NN}$ may approach points in $C$ far from
$P_C(z^k)$; in fact the computational cost of our inner loop is
lower than the computation of an inexact orthogonal projection of
$z^k$ onto $C$. On the other hand, it is not the case that any point
$z$ close enough to $P_C(z^k)$ will do the job. The crucial relation
for convergence of our method is $\|z^{k}_{0}-x\|\le \|z^k-x\|$ for
all $x\in C$, which may fail if we replace $z^{k}_{0}$ by points $z$
arbitrarily close to $P_C(z^k)$.

{\bf Algorithm A} is easily implementable, since $P_{W_{k,j}\cap C_{k,j}}$
and $P_{C_k}$, given in \eqref{ykj} and \eqref{algoritmo_1_paso_2} respectively,
have easy formula by Proposition \ref{proj_C(x)}. Hence, by
Proposition \ref{proj_C(x)} the projections onto $C_{k,j}\cap
W_{k,j}$ in {\bf Inner Loop} when used inside of {\bf Algorithm A} and $C_k$ in
\eqref{algoritmo_1_paso_2}, can be calculated explicitly. Therefore,
{\bf Algorithm A} may be considered as an explicit method, since it does not
solve hard subproblems.

\noindent We need the following boundedness assumptions on $\partial
c$.

\vspace{0.05in}

\noindent (H1) $\partial c$ is bounded on bounded sets.

In finite-dimensional spaces, this assumption is always satisfied in
view of Theorem $4.6.1$(ii) in \cite{iusem-regina}, due to the maximality of
$\partial c$. The maximality has been
proved in \cite{188}. For some equivalences with condition (H1), see
for instance Proposition $16.17$ in \cite{bauschke}. Moreover in the literature, (H1) has
been considered as the convergence analysis of
various methods solving optimization problems in
infinite-dimensional spaces; see, for instance,
\cite{AIS,yunier-iusem-4, poljak}. We only use this assumption for
establishing the well definition of {\bf Inner Loop}.

\vspace{0.05in}

\noindent (H2) Define
\begin{equation}\label{eta-k}\eta_k:=\max_{1\le i\le m}\{1,\|u_i^k\|\},\end{equation} with $u_i^k\in T_i(z_{i-1}^k)$.
Then, assuming that the stepsize sequence, $(\alpha_k)_{k\in \NN}$,
satisfies:
\begin{align}\label{betapaso}
\sum_{k=0}^\infty\alpha_k&=\infty,
\\ \label{betapaso1}
\sum_{k=0}^\infty(\eta_k\alpha_k)^2&<\infty.
\end{align}
We mention that in the analysis of \cite{nedic}, a stronger
condition than (H2) is required for proving convergence of the
incremental subgradient method. Recently, similar assumptions have
been used in the convergence analysis in \cite{bot, attouch}.

The condition \eqref{betapaso} (divergent-series) on the stepsizes
has been used widely for the convergence of classical projected
subgradient methods; see \cite{AIS,poljak}.

The condition \eqref{betapaso1} is used for establishing Proposition
\ref{cuasi-Fejer*2}, which implies
 boundedness of the sequence $(z^k)_{k\in \NN}$. When $(\alpha_k)_{k\in \NN}$ is in $\ell_2(\NN)$, the
condition \eqref{betapaso1} holds, assuming that the image of $T_i$,
$(i=1,2,\ldots,m)$ are bounded. Furthermore, it is possible to
assume a weaker sufficient condition for \eqref{betapaso1} as for
example: if $(\alpha_k)_{k\in \NN}=(1/k)_{k\in \NN}$, then the
sequence $(\eta_k)_{k\in \NN}$, defined in \eqref{eta-k}, may be
unlimited like the sequence $(k^s)_{k\in \NN}$ for any $s\in
(0,1/2)$. Finally, note that \eqref{eta-k} implies that $\eta_k\ge1$ and as consequence $\eta_k^2\ge \eta_k$, which combines with \eqref{betapaso1} implies
\begin{equation}\label{ponto-5}
\sum_{k=0}^\infty\eta_k\alpha_k^2<\infty.
\end{equation}

\subsection{Convergence Analysis }\label{convergencia}

Before establishing convergence of {\bf Algorithm A}, we need to ascertain
the validity of the stopping criterion as well as the fact that
{\bf Algorithm A} is well defined.
\begin{proposition}\label{propiedades_relaxed} Take $C$, $C_{k,j}$, $W_{k,j}$, $W_{k}$ and $C_k$ defined by {\bf Algorithm A}. Then,
\item [ {\bf (a)}] $C\subseteq C_{k,j}\cap W_{k,j}$, $C\subseteq C_k$ and
$z^k_i\in C_k$ for all $k$, $j$ and $i=0,1,\ldots,m$.
\item [ {\bf (b)}] {\bf Inner Loop} is well defined.
\end{proposition}
\begin{proof}
\quad (a):~It follows from \eqref{Ck} and the definition of the
subdifferential that $C\subseteq C_{k,j}$ for all $k$ and $j$. Note
that for all $y^{k,j}\notin C$, we have $\partial
f(y^{k,j})=\partial c(y^{k,j})$, thus $C\subseteq C_{k}$ by
\eqref{Ck}. Using
 Proposition $4$ and Corollary $1$ of
\cite{yunier-iusem-4}, with $C=\HH$, $f(x)=c^+(x):=\max\{0,c(x)\}$
implying that $f_*=0$, since our $C\neq \emptyset$, we get
$C\subseteq W_{k,j}$ for all $k$, $j$. By \eqref{ykj}, we have $z_0^k\in C_k$ and by
\eqref{algoritmo_1_paso_2}, $z^k_i\in C_k$ for all $k$,
$(i=1,\ldots,m)$.

\medskip

\noindent (b):~Regarding the projection step in \eqref{ykj} and
\eqref{algoritmo_1_paso_2}, item (a) shows that the projections onto
$C_{k,j}\cap W_{k,j}$ and $C_k$ are well defined. Using Theorem $2$
of \cite{yunier-iusem-4}, with $C=\HH$, $x^0=y^{k,0}$,
$f(x)=c^+(x):=\max\{0,c(x)\}$, we have that $(y^{k,j})_{j\in \NN}$
converges strongly to $P_C(y^{k,0})\in C$. Thus, {\bf Inner Loop} stops after finitely many steps. 
\end{proof}
\noindent A useful proposition for the convergence of algorithm is:
\begin{proposition}\label{propie}
Let $(z^k)_{k\in \NN}$ and $(z_i^k)_{k\in \NN}$, with $i=0,1,\ldots,
m$ be sequences generated by {\bf Algorithm A}. Then
\item [ {\bf (a)}] $\|z_j^k- z_i^k\|\leq (j-i)\eta_k\alpha_k$, for all $k\in \NN$ and $0\leq i\leq j \leq m$.
\item [ {\bf (b)}] For any $x\in C$ and $u\in T(x)$ such that
$u=\sum_{i=1}^m u_i$ with $u_i\in T_i(x)$, $(i=1,2,\ldots, m)$.
Then,
\begin{equation*}
\|z^{k+1}-x\|^2 \leq \|z^{k}-x\|^2+m\left[(\eta_k\alpha_k)^2+
(m-1)\eta\eta_k\alpha_{k}^2\right]-2\alpha_k \la u,z^{k}_{0}-x\ra,
\end{equation*} where $\eta:=\max_{1\le i\le m}{\|u_i\|}$.
\end{proposition}
\begin{proof}
\noindent (a): Since $z^k_i\in C_k$ for all $0\leq i \leq m$ and all
$k$, taking any  $u_j^k\in T_j(z_{j-1}^k) $ and using Lemma \ref{popiedades_projeccion}(a) and the
Cauchy-Schwarz inequality, we have
 \begin{eqnarray*}
  \|z^{k}_{j}-z^{k}_{i}\|&=& \left\|P_{C_k}\left(z^{k}_{j-1}-\alpha_k u_{j}^{k}\right)-P_{C_k}\left(z^k_i\right)\right\| \\
 &\le&  \|z^{k}_{j-1}-z^k_i-\alpha_k u^k_{j}\|
 \leq \|z^{k}_{j-1}-z^k_i \|+\|u_{j}^{k}\|\alpha_k \leq \cdots \leq
 (j-i)\,\eta_k\alpha_k.
\end{eqnarray*}

\noindent (b): Take $(x,u)\in \mbox{Gr}(T)$ with $u=\sum_{i=1}^{m}
u_i$ and $u_i\in T_i(x)$, $(i=1,2,\ldots, m)$. Using Lemma
\ref{popiedades_projeccion}(a) in the first inequality, and the
monotonicity of each component operator $T_i$ in the latter, we
obtain,
\begin{align}\nonumber
\dsty\| z_{i}^{k}-x\|^2&= \left\| P_{C_k}\left( z^{k}_{i-1}-\alpha_k
u_{i}^k\right)-P_{C_k}(x)\,\right\|^2\leq\left\|\left(
z_{i-1}^{k}-\alpha_k u_{i}^k\right)-x\right\|^2\\\nonumber &= \|
 z_{i-1}^{k}-x\|^2+(\|u_i^k\|\alpha_k)^2-2\alpha_k\,\,\la
u_{i}^k, z_{i-1}^{k}-x\ra\\ \nonumber &\leq\|
 z_{i-1}^{k}-x\|^2+(\|u_i^k\|\alpha_k)^2-2\alpha_k\,\,\la
u_{i}, z_{i-1}^{k}-x\ra, \ \ \ \ \nonumber
\end{align}
for all $k$ and $i=1,2,\ldots, m$. By summing the above inequalities
over $i=1,2,\ldots,m$ and using \eqref{eta-k}, we get
\begin{align}\nonumber
\dsty \|z^{k+1}-x\|^2&\leq\|
z^{k}_{0}-x\|^2+m(\eta_k\alpha_k)^2-\,2\, \alpha_k \sum_{i=1}^{m}\la
u_{i},z^{k}_{i-1}-x\ra\\\nonumber &=
\|z^{k}_{0}-x\|^2+m(\eta_k\alpha_k)^2-2\alpha_k
\sum_{i=1}^{m}\left(\la u_i,z^{k}_{0}-x\ra + \la
u_i,z^{k}_{i-1}-z^{k}_{0}\ra\right)\\\nonumber &=
\|z^{k}_{0}-x\|^2+m(\eta_k\alpha_k)^2-2\alpha_k \la u,z^{k}_{0}-x\ra
- 2\alpha_k \sum_{i=1}^{m}\la u_i,z^{k}_{i-1}-z^{k}_{0}\ra.\nonumber
\end{align}
Using the Cauchy-Schwarz inequality and item (a) for $j=m$ and
$i=0$,  we have
\begin{align*}
 \dsty \|z^{k+1}-x\|^2&\leq \|z^{k}_{0}-x\|^2+m(\eta_k\alpha_k)^2-2\alpha_k \la u,z^{k}_{0}-x\ra+2\alpha_k
 \sum_{i=1}^{m}\|u_i\|\|z^{k}_{i-1}-z^{k}_{0}\|\\ \nonumber &\leq \|z^{k}_{0}-x\|^2+m(\eta_k\alpha_k)^2
 + 2 \alpha_{k}^2\sum_{i=1}^{m}(i-1)\|u_i\|\eta_k-2\alpha_k \la u,z^{k}-x\ra\\&\le \|z^{k}-x\|^2+m(\eta_k\alpha_k)^2
 + m(m-1)\eta\eta_k\alpha_{k}^2-2\alpha_k \la u,z^{k}_{0}-x\ra,
\end{align*}
where the last inequality is a direct consequence of the fact that
$z^{k}_{0}$ is obtained by {\bf Inner Loop} and
defining $\eta=\max_{1\le i\le m}{\|u_i\|}$, we prove the proposition.\end{proof}
We continue by proving the quasi-Fej\'er properties of the sequences
$(z^k)_{k\in \NN}$ generated by {\bf Algorithm A}. From now on, we assume
that the solution set, $S_*$, of problem \eqref{problem} is
nonempty.
\begin{proposition}\label{cuasi-Fejer*2}
The sequences $(z^k)_{k\in \NN}$ are quasi-Fej\'er convergent to
$S_*$.
\end{proposition}
\begin{proof}
Take $\bar{x}\in S_*$. Then, there exists $\bar{u}\in T(\bar{x})$
such that
\begin{equation}\label{tomando_u_barra}
\la\bar{u},x-\bar{x}\ra\geq 0 \quad \forall\, x\in C,
\end{equation}
where $\bar{u}=\sum_{i=1}^{m} \bar{u}_i$, with $\bar{u}_i \in
T_i(\bar{x})$, $(i=1,2,\ldots, m)$. Using now Proposition
\ref{propie}(b) and taking $\bar{\eta}=\max_{1\le i\le
m}{\|\bar{u}_i\|}$, we get
 \begin{align}\nonumber \|z^{k+1}-\bar{x}\|^2 &\leq \|z^{k}-\bar{x}\|^2+m\left[(\eta_k\alpha_k)^2+
 (m-1)\bar{\eta}\eta_k\alpha_{k}^2\right]
 -2\alpha_k \la \bar{u},z^{k}_{0}-x\ra \\ \nonumber &=\|z^{k}-\bar{x}\|^2+m\left[(\eta_k\alpha_k)^2+
 (m-1)\bar{\eta}\eta_k\alpha_{k}^2\right]
 -2\alpha_k ( \la \bar{u},z^{k}_{0}-P_{C}(z^{k}_{0})\ra + \la \bar{u},P_{C}(z^{k}_{0})- \bar{x}\ra )\\ \nonumber
&\leq \|z^{k}-\bar{x}\|^2+m\left[(\eta_k\alpha_k)^2+
 (m-1)\bar{\eta}\eta_k\alpha_{k}^2\right]+2\alpha_k\|\bar{u}\|\mbox{dist}(z^k_0,C)
\\ \label{13'} &\leq \|z^{k}-\bar{x}\|^2+m\left[(\eta_k\alpha_k)^2+
 (m-1)\bar{\eta}\eta_k\alpha_{k}^2\right]+2\theta\|\bar{u}\|\alpha_k^2,
\end{align}
where we used \eqref{tomando_u_barra} and the Cauchy-Schwarz
inequality in the second inequality and the last inequality is a
consequence of the fact that $z^{k}_{0}$ is obtained by {\bf Inner
Loop}. It follows from \eqref{13'},
 \eqref{betapaso1} and \eqref{ponto-5} that $(z^k)_{k\in \NN}$ is quasi-Fej\'er
convergent to $S_*$.
\end{proof}
\noindent Next we establish some convergence properties of {\bf Algorithm
A}.
\begin{proposition}\label{propiedades_relaxed-2} Let $(z^k)_{k\in \NN}$ and $(x^k)_{k\in \NN}$ be the sequences generated
by {\bf Algorithm A}. Then,
\item [ {\bf (a)}]$\dsty x^{k+1}=\frac{1}{\sigma_k}\sum_{j=0}^k
\alpha_j z^{j+1}$, for all $k$;

\item [ {\bf (b)}]$(x^k)_{k\in \NN}$ are bounded;

\item [ {\bf (c)}]$\lim_{k\rightarrow\infty}{\rm dist}(x^k,C)=0$;

\item [ {\bf (d)}] all weak cluster points of $(x^k)_{k\in \NN}$ belong to $C$.
\end{proposition}
\begin{proof}

\noindent (a):~ We proceed by induction on $k$. For $k=0$, using
\eqref{algoritmo_2_paso_3-bonito} and that $\sigma_0=\alpha_0$, we
have that $x^1=z^{1}$. By hypothesis of induction, assume that

\begin{equation}\label{hip-ind}x^k=\dsty\frac{1}{\sigma_{k-1}}\sum_{j=0}^{k-1}\alpha_j z^{j+1}.\end{equation}
\noindent Using \eqref{16'} and \eqref{algoritmo_2_paso_3-bonito},
we obtain
$$
x^{k+1}=\frac{\sigma_{k-1}}{\sigma_{k}}x^k+\frac{\alpha_{k}}{\sigma_{k}}z^{k+1}.
$$ By \eqref{hip-ind} and the above equation, we get
$$x^{k+1}=\dsty\frac{1}{\sigma_{k}}\sum_{j=0}^{k-1}\alpha_j
z^{j+1}+\dsty\frac{\alpha_{k}}{\sigma_{k}}z^{k+1}=\dsty\frac{1}{\sigma_{k}}\sum_{j=0}^{k}\alpha_j
z^{j+1},$$ proving the assertion.

\medskip

\noindent (b):~ Using Proposition \ref{cuasi-Fejer*2} and
Proposition \ref{cuasi-Fejer}(a), we have the boundedness of
$(z^k)_{k\in \NN}$. We assume that there exists $R>0$ such that
$\|z^k\|\leq R$, for all $k$. By the previous item, $$\|x^k\|\leq
\dsty\frac{1}{\sigma_{k-1}}\sum_{j=0}^{k-1}\alpha_j\|z^{j+1}\|\leq
R,$$ for all $k$.

\medskip

\noindent (c):\; It follows from definition of $z^{k}_{0}$ that
\begin{equation}\label{distancia-menor-beta-k}{\rm dist}(z^{k}_{0},C)\leq\theta\alpha_k.\end{equation}
\noindent Define
\begin{equation}\label{27}
\tilde{x}^{k+1}:=\frac{1}{\sigma_k}\sum_{j=0}^k\alpha_j
P_C({z^{j}_0}).
\end{equation}
\noindent Since $\dsty\frac{1}{\sigma_k}\sum_{j=0}^k\alpha_j=1$ by
\eqref{16'}, we get from the convexity of $C$, that
$\tilde{x}^{k+1}\in C$. Thus,
\begin{align}\label{22}\nonumber
{\rm dist}(x^{k+1},C)&\leq
\|x^{k+1}-\tilde{x}^{k+1}\|=\left\|\frac{1}{\sigma_k}\sum_{j=0}^k\alpha_j\left(z^{j+1}-P_C({z^{j}_0})\right)\right\|\leq\frac{1}{
\sigma_k}\sum_{j=0}^k\alpha_j\left \|z^{j+1}-P_C(z^{j}_0)\right\| \\
&\le \frac{1}{
\sigma_k}\sum_{j=0}^k\alpha_j\left(\|z^{j+1}-z^{j}_0\| +
\|z^{j}_0-P_C(z^{j}_0)\|\right)\le \frac{1}{\sigma_k}\sum_{j=0}^k
\alpha_j(m\, \eta_j\alpha_j+ {\rm dist }(z^{j}_0,C))\nonumber\\&\leq
\frac{1}{\sigma_k}\sum_{j=0}^k (m\,
\eta_j\alpha_j^2+\theta\alpha_j^2)\le\frac{m+\theta}{\sigma_k}\sum_{j=0}^\infty \,
\eta_j\alpha_j^2,
\end{align}
using the fact that $\tilde{x}^{k+1}$ belongs to $C$ in the first
inequality, (b) and \eqref{27} in the equality, convexity of
$\|\cdot\|$ in the second inequality, Proposition \ref{propie}(a),
with $j=m$ and $i=0$, in the fourth inequality,
\eqref{distancia-menor-beta-k} in the fifth inequality and that $\eta_j\ge1$ for all $j$ in the last one. Taking limits in
\eqref{22} and using \eqref{16'}, \eqref{betapaso} and \eqref{ponto-5}, we get
$\lim_{k\rightarrow\infty}{\rm dist}(x^{k+1},C)=0$, establishing
(c).

\medskip


\noindent (d):\; Follows directly from (c).
\end{proof}

\noindent Next we prove optimality of the cluster points of
$(x^k)_{k\in \NN}$.

\begin{theorem}\label{todos-ptos-de-acumulacion-son-solucion}
All weak cluster points of the sequence $(x^k)_{k\in \NN}$ generated
by {\bf Algorithm A} solve problem \eqref{problem}.
\end{theorem}
\begin{proof}
Using Proposition \ref{propie}(b), we get, for any $x\in C$, $u\in
T(x)$ and for all $j>0$,
\begin{align}\label{222}\nonumber
\|z^{j+1}-x\|^2 &\leq \|z^{j}-x\|^2+m\left[(\eta_j\alpha_j)^2+
(m-1)\eta\eta_j\alpha_{j}^2\right]-2\alpha_j \la u,z^{j}_{0}-x\ra,
\\ \nonumber &= \|z^j-x\|^2 +m\left[(\eta_j\alpha_j)^2+
(m-1)\eta\eta_j\alpha_{j}^2\right]-2\alpha_j \la u,z^{j}_{0}-z^{j+1}\ra-2\alpha_j\la u,z^{j+1}-x \ra \\
\nonumber &\leq \|z^j-x\|^2 + m\left[(\eta_j\alpha_j)^2+
(m-1)\eta\eta_j\alpha_{j}^2\right]+2\alpha_j \| u\|\|z^{j}_{0}-z^{j+1}\|-2\alpha_j\la u,z^{j+1}-x \ra \\
&\leq \|z^j-x\|^2 +m\left[(\eta_j\alpha_j)^2+
(m-1)\eta\eta_j\alpha_{j}^2\right]+ 2m\|u\|\eta_j\alpha_j^2
-2\alpha_j\la u,z^{j+1}-x \ra,
\end{align} using the Cauchy-Schwarz inequality in the second inequality and Proposition \ref{propie}(a), with $j=m$ and $i=0$, in the last one.
Rewriting and summing \eqref{222} from $j=0$ to $j=k$ and dividing
by $\sigma_{k}$, we obtain from Proposition
\ref{propiedades_relaxed-2}(a) that
\begin{align*}\label{desigualdad_sumada1}
\frac{1}{\sigma_{k}}\sum_{j=0}^{k}\left(\|z^{j+1}-x
\|^2-\|z^j-x\|^2-m\left[(\eta_j\alpha_j)^2-(m-1)\eta\eta_j\alpha_{j}^2
- 2\|u\|\eta_j\alpha_j^2\right]\right)\leq 2\la u,x-x^{k+1}\ra.
\end{align*}
After rearrangements, we have
\begin{equation}\label{desigualdad_sumada}
\frac{1}{\sigma_{k}}\big(\|z^{k+1}-x
\|^2-\|z^0-x\|^2-m\sum_{j=0}^\infty\left[(\eta_j\alpha_j)^2-(m-1)\eta\eta_j\alpha_{j}^2
- 2\|u\|\eta_j\alpha_j^2\right]\big)\leq 2\la u,x-x^{k+1}\ra.
\end{equation}
Let $\hat{x}$ be a weak cluster point of $(x^k)_{k\in \NN}$.
Existence of $\hat{x}$ is guaranteed by Proposition
\ref{propiedades_relaxed-2}(b). Note that $\hat{x}\in C$ by
Proposition \ref{propiedades_relaxed-2}(d). Taking limits in
\eqref{desigualdad_sumada}, using the boundedness of $(z^k)_{k\in
\NN}$ by Proposition \ref{cuasi-Fejer*2} and \eqref{betapaso}-\eqref{ponto-5}, we
obtain that $\la u,x-\hat{x}\ra\geq0$ for all $x\in C$ and $u\in
T(x)$. Using Lemma \ref{consecuencia-continuidad}, we get that
$\hat{x}\in S_*$. Hence, all weak cluster points of $(x^k)_{k\in
\NN}$ solve problem \eqref{problem}.
\end{proof}
\noindent We now  state and prove the weak convergence of the main
sequence generated by {\bf Algorithm A}.
\begin{theorem}\label{convergencia-a-la-solucion}
Define $x_*=\lim_{k\rightarrow\infty} P_{S_*}(z^k)$. Then
$(x^k)_{k\in \NN}$ converges weakly to $x_*$.
\end{theorem}
\begin{proof}
Define $p^k:=P_{S_*}(z^k)$ the orthogonal
projection of $z^k$ onto $S_*$. Note that $p^k$ exists, since the solution set $S_*$
is nonempty, closed and convex by Lemma \ref{cond_limitacion}. By
Proposition \ref{cuasi-Fejer}, $(z^k)_{k\in \NN}$ is quasi-Fej\'er
convergent to $S_*$. Therefore, it follows from Lemma
\ref{distancia-q-F-convergente} that $(P_{S_*}(z^k))_{k\in \NN}$ is
strongly convergent. Set
\begin{equation}\label{uk-converge-a-x*}x_*:=\lim_{k\rightarrow\infty} P_{S_*}(z^k)=
\lim_{k\rightarrow\infty} p^k.\end{equation} By Proposition
\ref{propiedades_relaxed-2}(b), $(x^k)_{k\in \NN}$ is bounded and by
Theorem \ref{todos-ptos-de-acumulacion-son-solucion} each of its
weak cluster points belong to $S_*$. Let $(x^{i_k})_{k\in \NN}$ be
any weakly convergent subsequence of $(x^k)_{k\in \NN}$, and let
$\bar{x}\in S_*$ be its weak limit. In order to establish the weak convergence of $(x^k)_{k\in
\NN}$, it suffices to show that
$\bar{x}=x_*$.

By Lemma \ref{popiedades_projeccion}(b) we have that
$\la\bar{x}-p^j,z^j-p^j\ra\leq0$ for all $j$. Let $\xi=\sup_{0\leq
j\leq\infty}\|z^j-p^j\|$. Since $(z^k)_{k\in \NN}$ is bounded by
Proposition \ref{cuasi-Fejer}(a), we get that $\xi<\infty$. Using
the Cauchy-Schwarz inequality,
\begin{equation}\label{b}
\la\bar{x}-x_*,z^j-p^j\ra\leq\la p^j-x_*,z^j-p^j\ra\leq
\xi\,\|p^j-x_*\|,
\end{equation}
for all $j$. Multiplying \eqref{b} by
$\displaystyle\frac{\alpha_{j-1}}{\sigma_{k-1}}$ and summing from
$j=1$ to $k-1$, we get from Proposition
\ref{propiedades_relaxed-2}(a),
\begin{equation}\label{unicidad}
\left\la\bar{x}-x_*,x^{k}-\frac{1}{\sigma_{k-1}}\sum_{j=1}^{k-1}
\alpha_{j-1} p^j\right\ra\leq
\frac{\xi}{\sigma_{k-1}}\sum_{j=1}^{k-1}\alpha_{j-1}\|p^j-x_*\|.
\end{equation}
\noindent Define
$$\zeta_{k,j}:=\frac{\alpha_j}{\sigma_{k}} \ \ \ \ (k\ge 0,\ \ \ 0\le j\le k).$$
It follows from the definition of $\sigma_k$, that
$\lim_{k\rightarrow\infty}\zeta_{k,j}=0$ for all $j$ and
$\sum_{j=0}^k\zeta_{k,j}=1$ for all $k$. Using
\eqref{uk-converge-a-x*} and Proposition
\ref{conv-convinacion-de-terminos} with $\dsty
w^k=\sum_{j=1}^{k}\zeta_{k-1,j-1}p^{j}=\frac{1}
{\sigma_{k-1}}\sum_{j=0}^{k-1}\alpha_j p^{j+1}$, we have
\begin{equation}\label{conv-convexa-converge-c}
x_*=\lim_{k\to \infty}p^k=\lim_{k\to
\infty}\frac{1}{\sigma_{k}}\sum_{j=0}^{k}\alpha_jp^{j+1},
\end{equation}
and
\begin{equation}\label{conv-convexa-converge-b}
\lim_{k\rightarrow\infty}\frac{1}{\sigma_{k}}\sum_{j=0}^{k}\alpha_j\|p^{j+1}-x_*\|=0.
\end{equation}
Taking limits in \eqref{unicidad} over the subsequence $(i_k)_{k\in
\NN}$, and using \eqref{conv-convexa-converge-c} and
\eqref{conv-convexa-converge-b}, we get
$\left\la\bar{x}-x_*,\bar{x}-x_*\right\ra\leq0,$ implying that
$\bar{x}=x_*$.
\end{proof}
\section{\bf Final Remarks}\label{remarks}
In this section we discuss the assumptions of our scheme,
showing examples as well as some alternatives for changing these assumptions.

One problem in establishing the well definition of the sequence
generated by {\bf Algorithm A} may be the difficulty of choosing stepsizes
satisfying Assumption~(H2); see \eqref{betapaso}-\eqref{betapaso1}.
This is clear in the important case, where the operators have
bounded range, the sequence $(\eta_k)_{k\in \NN}$, defined in
\eqref{eta-k}, is bounded. Hence, any sequence $(\alpha_k)_{k\in
\NN}$ in $\ell_2(\NN)\setminus\ell_1(\NN)$ may be used satisfying
(H2). 

Now we present some examples showing that (H2) is verified for
some different instances.
\begin{example}\label{Ex1}
Consider the variational inequality problem in a Hilbert space
$\HH$, for $T$ and $S=\argmin_{x\in \HH} f(x)$, where
$T:\HH\rightrightarrows\HH$ is a maximal monotone operator and $f:
\HH\rightarrow\RR$ is a continuous convex function which satisfies
(H1). This problem is equivalent to problem \eqref{problem}, with
$m=1$, $T_1=T$, $c=f-f_*$, where $f_*=\min_{x\in \HH} f(x)$ and
Algorithm $A$ may be rewritten as follows:
\begin{center}\fbox{\begin{minipage}[b]{\textwidth}
\begin{AlgA1} Given $\theta>0$

\noindent{\bf Initialization step:} Take $ x^0\in \HH$. Define $z^0:=x^0$ and $\sigma_{0}:=\beta_0$.

\medskip
\noindent {\bf Iterative step:} Given $z^k$. If $f(z^k)\le f_*$,
then take $z^k_0:=z^k$ and $S_k=\{x\in \HH:\,\la
g^{k},x-z^k_0\ra\leq 0\}$ where $g^k\in \partial f(z^k_0)$. Else,
compute $(z^k_0,S_k):={\bf Inner Loop}(z^k,\theta,\beta_k)$ like {\bf Algorithm A}. Set
\begin{equation*}
z^{k+1}=P_{S_k}\left(z_0^{k}-\frac{\beta_k}{\eta_k} u^{k}\right),
\end{equation*}
where $u^{k}\in T(z_0^{k})$ and $\eta_k=\max\{1,\|u^{k}\|\}$. Define
\begin{equation*}\sigma_{k}:=\sigma_{k-1}+\frac{\beta_k}{\eta_k},\end{equation*}
\begin{equation*}
x^{k+1}:=\left(1-\frac{\beta_k}{\sigma_{k}}\right)x^{k}+\frac{\beta_k}{
\sigma_{k}} z^{k+1}.
\end{equation*}
\end{AlgA1}\end{minipage}}\end{center}
\end{example}
\noindent {\bf Algorithm A1} is the point-to-set version of the algorithm proposed
in \cite{yunier-iusem-2}. Assuming that the problem has solutions
and that $(\beta_k)_{k\in \NN}$ in
$\ell_2(\NN)\setminus\ell_1(\NN)$, the analysis of the convergence
follows directly from the analysis in \cite{yunier-iusem-2}. (The
inner loop is slightly different however the convergence proof
remains essentially unchanged.)
\begin{example}\label{Ex2}
We consider the optimization problem of the form
\begin{equation}\label{prob-ex2}
\min_{x\in X} \quad \phi_1(L(x))+\phi_2(x),
\end{equation}
where $L:X\rightarrow Y$ is a continuous linear operators, with
closed range, $X$ and $Y$ are two Hilbert spaces and $\phi_1:
Y\rightarrow \RR$, $\phi_2: X\rightarrow \RR$ are convex and
continuous functions.

This is a classical problem which appears in many applications in
mechanics and economics; see, for instance, \cite{gabay}. Denote
$K=\{(x,y)\in X\times Y\,:\, L(x)-y=0\},$ and
$$A=\left(\begin{array}{lr}\partial \phi_1& 0\\ 0& 0\end{array}\right) \quad \quad B=\left(\begin{array}{lr} 0& 0\\ 0& \partial \phi_2\end{array}\right).$$
$A$ and $B$ are maximal monotone and \eqref{prob-ex2} is
equivalent to problem \eqref{problem}, with $m=2$, $T_1=A$, $T_2=B$
and $C=K$ in $\HH=X\times Y$.

In this case our algorithm does not require the inner loop, since
the feasible set $K$ is a linear and closed subspace. Thus, the
projection onto $K$ is easy to compute; in effect, the set $K$ can
be rewritten as $$K=\{(x,y)\in X\times Y\,:\,
c(x,y)=\frac{1}{2}\|L(x)-y\|^2\le 0\},$$ and $\nabla
c(x,y)=\left(%
\begin{array}{c}
 L^*(L(x)-y) \\
  y-L(x) \\
\end{array}%
\right)$ and hence,
$$P_K(x,y)=(x,y)-\frac{1}{2}\|L(x)-y\|^2\frac{\nabla
c(x,y)}{\|\nabla c(x,y)\|^2}.$$ For this example {\bf Algorithm A} may
be rewritten as follows:

\begin{center}\fbox{\begin{minipage}[b]{\textwidth}
			\begin{AlgA2} 

\item []\noindent{\bf Initialization step:} Take $ x^0:=(x_1^0,x_2^0)\in K$. Define $z^0:=x^0$ and $\sigma_{0}:=\beta_0$.

\medskip

\noindent {\bf Iterative step:} Given $z^k=(z_1^{k},z_2^{k})$.
Compute

\begin{equation*}
(z^{k+1}_{1,1}\,,\,z^{k+1}_{1,2})=P_{K}\left(z_1^{k}-\frac{\beta_k}{\eta_k}
u_1^{k}\,,\,z_2^{k}\right),
\end{equation*}
and
\begin{equation*}
(z^{k+1}_{2,1}\,,\,z^{k+1}_{2,2})=P_{K}\left(z^{k+1}_{1,1}\,,\,z^{k+1}_{1,2}-\frac{\beta_k}{\eta_k}
u_2^{k}\right),
\end{equation*}
where $u_{1}^{k}\in \partial \phi_1 (z_1^{k})$, $u_{2}^{k}\in
\partial \phi_2 (z_{1,2}^{k+1})$ and
$\eta_k=\max\{1,\|u^{k}_{1}\|\,\|u^{k}_{2}\|\}$. Set
$z^{k+1}=(z_{2,1}^{k+1},z^{k+1}_{2,2})$, and define
\begin{equation*}\sigma_{k}:=\sigma_{k-1}+\frac{\beta_k}{\eta_k},\end{equation*}
\begin{equation*}
x_1^{k+1}:=\left(1-\frac{\beta_k}{\sigma_{k}}\right)x_1^{k}+\frac{\beta_k}{
\sigma_{k}} z^{k+1}_{2,1}.
\end{equation*}
and
\begin{equation*}
x_2^{k+1}:=\left(1-\frac{\alpha_k}{\sigma_{k}}\right)x_2^{k}+\frac{\beta_k}{
\sigma_{k}} z^{k+1}_{2,2}.
\end{equation*}
Set
\begin{equation*}
x^ {k+1}=\left(x_1^{k+1},x_2^{k+1}\right).
\end{equation*}
\end{AlgA2}\end{minipage}}\end{center}

\end{example}

\begin{example}\label{Ex3}
Consider the minimax problem:
\begin{equation}\label{prob-ex3}
\min_{x_1\in X} \max_{x_2\in X} \{ \phi_1(x_1)-\phi_2(x_2)+\la x_2,
L(x_1)\ra\},
\end{equation}
where $L:X\rightarrow X$ is a self adjoint and continuous linear
operator, $X$ is a Hilbert space, $\phi_1: X\rightarrow \RR$,
$\phi_2: X\rightarrow \RR$ are convex and continuous functions and
$\phi_2$ is G\^ateaux differentiable. This problem was presented in
\cite{Rock} and under a suitable constraint qualification, this
problem is equivalent to problem \eqref{problem}, with $m=2$,
$\HH=C=X\times X$, and $$T_1(x_1,x_2)=A(x_1,x_2)=(\partial
\phi_1(x_1),0)$$ and $$T_2(x_1,x_2)=B(x_1,x_2)=(L(x_2),\nabla
\phi_2(x_2)-L(x_1)),$$ which are maximal monotone operators.
{\bf Algorithm A} can be rewritten as follows:

\begin{center}\fbox{\begin{minipage}[b]{\textwidth}
			\begin{AlgA3} 
				
				\item []\noindent{\bf Initialization step:} Take $ x^0:=(x_1^0,x_2^0)\in
X\times X$. Define $z^0:=x^0$ and
$\sigma_{0}:=\beta_0$.

\medskip

\noindent {\bf Iterative step:} Given $z^k=(z_1^{k},z_2^{k})$.
Compute
\begin{align*}
z^{k+1}_{1,1}=z_1^{k}-\frac{\beta_k}{\eta_k}u_1^{k}\qquad &
z^{k+1}_{2,1}=  z_2^{k},\\
z^{k+1}_{1,2}=z^{k+1}_{1,1}-\frac{\beta_k}{\eta_k} L(z^{k+1}_{2,1})
\qquad &\dsty z^{k+1}_{2,2}=
z_{2,1}^{k}-\frac{\beta_k}{\eta_k}\left(\nabla \phi_2(z^{k+1}_{2,1})
-L(z^{k+1}_{1,1})\right),
\end{align*}
where $u_{1}^{k}\in \partial \phi_1 (z_1^{k})$ and
$\eta_k=\max\left\{1,\|u^{k}_{1}\|,\|L(z^{k+1}_{2,1})\|,\|\nabla
\phi_2(z^{k+1}_{2,1}) -L(z^{k+1}_{1,1})\|\right\}$. Set
$z^{k+1}:=(z_{1}^{k+1},z^{k+1}_{2})=(z_{1,2}^{k+1},z^{k+1}_{2,2})$,
and define
\begin{equation*}\sigma_{k}:=\sigma_{k-1}+\frac{\beta_k}{\eta_k},\end{equation*}
\begin{equation*}
x_1^{k+1}:=\left(1-\frac{\beta_k}{\sigma_{k}}\right)x_1^{k}+\frac{\beta_k}{
\sigma_{k}} z^{k+1}_1.
\end{equation*}
\begin{equation*}
x_2^{k+1}:=\left(1-\frac{\beta_k}{\sigma_{k}}\right)x_2^{k}+\frac{\beta_k}{
\sigma_{k}} z^{k+1}_2.
\end{equation*}
Set
\begin{equation*}
x^ {k+1}=(x_1^{k+1},x_2^{k+1}).
\end{equation*}
\end{AlgA3}\end{minipage}}\end{center}
\end{example}

\noindent In {\bf Algorithm A2} and {\bf Algorithm A3}, presented in the above
examples, the stepsize is $\dsty \alpha_k=\beta_k/\eta_k$,
for all $k$ and hence $\eta_k\alpha_k=\beta_k$ and \eqref{betapaso1}
is equivalent to $(\beta_k)_{k\in \NN}$ in $\ell_2(\NN)$, which by
Proposition \ref{cuasi-Fejer*2} implies, the boundedness of the
sequence $(z^k)_{k\in \NN}$. Moreover as a consequence of
Proposition \ref{propie}(a) the sequences $(z_{i-1}^k)_{k\in \NN}$,
$(i=1,2,\ldots,m )$ are bounded. Now, condition \eqref{betapaso}
becomes in
\begin{equation}\label{nueva-beta-paso}\dsty\sum_{k=0}^\infty\frac{\beta_k}{\eta_k}=\infty,\end{equation}
where $\eta_k=\max_{1\le i\le m}\{1,\|u_i^k\|\}$, with $u_i^k\in
T_i(z_{i-1}^k)$. Thus, a sufficient condition for
\eqref{nueva-beta-paso} is that the image of $T_i$,
$(i=1,2,\ldots,m)$ is bounded on bounded sets (since the sequences
$(z_{i-1}^k)_{k\in \NN}$, $(i=1,2,\ldots,m )$ are bounded).
Moreover, $(\eta_k)_{k\in \NN}$ may be an unlimited sequence as for
example $(k^s)_{k\in \NN}$ for any $s\in (0,1/2)$.

Therefore Assumption~(H2) turns into ``$(\beta_k)_{k\in \NN}$ lies
in $\ell_2(\NN)\setminus\ell_1(\NN)$'', which is a requirement widely used in the literature. The convergence analysis of
{\bf Algorithm A2} and {\bf Algorithm A3} follows from the convergence analysis of
{\bf Algorithm A}.

Another important point on {\bf Algorithm A}, is that {\bf Inner Loop} uses the distance function. It is clear that
this is weakest that compute the exact projection for almost all
instances. Furthermore inside of the inner loop, we may only check
the condition related with the distance on selected index. In
connection we may include the following assumption:
\begin{enumerate}
\item[(H3)] Assume that a Slater point is available, i.e. there exists a point $w\in \HH$ such that $c(w)<0$.
\end{enumerate}
\noindent If Assumption~(H3) holds, by Lemma \ref{conjetura} we can
obtain an explicit algorithm for a quite general convex set $C$,
replacing the inequality ${\rm dist}(y^{k,j+1}, C)\leq \theta
\alpha_k$ in {\bf Inner Loop} on {\bf Algorithm A} by
$\tilde{c}(y^{k,j+1})\leq \theta \alpha_k$, where
$$\tilde{c}(x)=\left\{
\begin{array}{cc}
\dsty\frac{\|x-w\|c(x)}{c(x)-c(w)} & \mbox{if} \,\, x\notin C \\ \\
0 & \mbox{if} \,\, x\in C.\end{array} \right.$$ All our convergence
results are preserved. (H3) is a hard assumption in Hilbert spaces
and the point $w$ is almost always unavailable. Hence, such
assumptions can be replaced by a rather weaker one, namely:

\begin{enumerate}
\item [(H3$^*$)] There exists an easily computable and continuous
$\tilde{c}:\HH\rightarrow\RR$, such that ${\rm dist}(x,C)\le
\tilde{c}(x)$ for all $x\in \HH$, and $\tilde{c}(x)=0$ if and only
if $c(x)=0$.
\end{enumerate}
There are examples of sets $C$ for which no Slater point is
available, while (H3$^*$) holds, including instances in which $C$
has an empty interior. An exhaustive discussion about weak
constraint qualifications for getting error-bound can be found in
\cite{sien, lewis}. 

Finally, regarding the complexity of the proposed inner loop in {\bf Algorithm A}. To stop the inner loop at an $\alpha$-solution, in the worse case, we need to do
$\mathcal{O}\left(\alpha^{-2}\right)$ iterations; see Section $3.2.1$ of \cite{Book-Nest}. Emphasizing that here we work with a general nonsmooth and convex function $c$ and to get better performance of our inner loop we have to assume, at least, differentiability of $c$; see, for instance, \cite{nesterov-1983} and the references therein.

\subsubsection*{ACKNOWLEDGMENT}

\noindent JYBC and RDM were partially supported by project
CAPES-MES-CUBA 226/2012. JYBC was partially supported by CNPq grants
303492/2013-9, 474160/2013-0 and 202677/2013-3 and by project
UNIVERSAL FAPEG/CNPq. RDM was supported by his scholarship for his
doctoral studies, granted by CAPES.

This work was completed while the first author was visiting the
University of British Columbia. The author is very grateful for the
warm hospitality. The authors would like to thank to Professor Dr.
Heinz H. Bauschke, Professor Dr. Ole Peter Smith and anonymous
referees
whose suggestions helped us to improve the presentation of
this paper.

\bibliographystyle{plain}

\end{document}